\documentclass[reqno,11pt]{amsart}
\usepackage[foot]{amsaddr}
\usepackage{bbold}
\usepackage{charter}
\usepackage[margin=1in]{geometry}
\usepackage[colorlinks=true,linktoc=all,linkcolor=blue,citecolor=blue]{hyperref}

\theoremstyle{plain}
\newtheorem{theorem}{Theorem}[section]
\newtheorem*{theorem*}{Theorem}
\newtheorem{prop}[theorem]{Proposition}
\newtheorem{lemma}[theorem]{Lemma}
\newtheorem{cor}[theorem]{Corollary}
\theoremstyle{definition}

\newtheorem*{definition*}{Definition}

\numberwithin{equation}{section}
\everymath{\displaystyle}
\allowdisplaybreaks

\newcommand{\D}{\mathbb D}

\newcommand{\A}{\alpha}

\newcommand{\Ph}{\varphi}
\newcommand{\CO}{C_{\Ph}}

\title[Isometric Composition Operators]{Isometric Composition Operators on the Analytic Besov Spaces}

\author{Robert F.~Allen\textsuperscript{1}, Katherine C.~Heller\textsuperscript{2}, and Matthew A.~Pons\textsuperscript{2}}
\address{\textsuperscript{1}Department of Mathematics and Statistics, University of Wisconsin-La Crosse}
\address{\textsuperscript{2}Department of Mathematics, North Central College}

\email{rallen@@uwlax.edu, kheller@noctrl.edu, mapons@noctrl.edu}


\begin{document}

\begin{abstract}
We investigate the isometric composition operators on the analytic Besov spaces.  For $1<p<2$ we show that an isometric composition operator is induced only by a rotation of the disk.  For $p>2$, we extend previous work on the subject.  Finally, we analyze this same problem for the Besov spaces with an equivalent norm. 
\end{abstract}

\maketitle

\section{Introduction}

Classifying the form of isometries dates back to Banach who characterized these types of operators on certain $L^p$ spaces and $C(X)$, where $X$ is a compact metric space, in \cite{SB}. There has also been much work in identifying general isometries on spaces of analytic functions.  For the $H^p$ spaces, $1\leq p<\infty$, $p\neq 2$, Forelli \cite{FF} showed that isometries take the form of weighted composition operators.  Rudin extended this work to several variables in \cite{WR}, Kolaski showed that a similar result holds on the Bergman and weighted Bergman spaces in \cite{CKR} and \cite{CKW}, and El-Gebeily and Wolfe considered the case of the disk algebra in \cite{EW}.  Horner and Jamison \cite{HJMAPS} improved upon Kolaski's results for the Bergman spaces by providing a more detailed analysis of the self-maps described by Kolaski, and also investigated situations where the isometries do not take the form of weighted composition operators in \cite{HJ}.

In our investigation we will focus on classifying isometries among a specific class of operators.  Let $\Ph$ be an analytic self-map of the unit disk $\D$ and let $\mathcal{Y}$ be a space of functions analytic on $\D$.  We then define the composition operator $\CO$ by $\CO f=f\circ\Ph$ for all $f\in\mathcal{Y}$.  If in addition $\psi$ is analytic on $\D$, then we define the weighted composition operator $W_{\psi,\Ph}$ by $W_{\psi,\Ph}f=\psi(f\circ\Ph)$ for $f\in \mathcal{Y}$.  In the study of composition operators and weighted composition operators, the goal is to relate the operator properties of $\CO$ or $W_{\psi,\Ph}$ to the function theoretic properties of the symbol $\Ph$, or $\psi$ and $\Ph$, respectively.  For general information on composition operators we point the reader to \cite{CMC}.

The first results explicitly concerning composition operators appeared in Nordgren's paper \cite{EN}.  For the Hardy space $H^2$, Nordgren showed that any inner function $\Ph$ which fixes the origin induces an isometric composition operator; the converse to this is also true (see \cite{JS}).  A proof extending this result to the full range of $H^p$ spaces, $1\leq p<\infty$, was recently given in \cite{MVB}, though the authors do comment that the result may have been known to experts in the field.  In contrast to this, for the standard Bergman spaces $A^p$, $1\leq p<\infty$, it was shown in \cite{CH} that the only isometric composition operators arise from rotations of the disk.  The article \cite{MVB} extends this to weighted Bergman spaces $A_{\omega}^p$, where $1\leq p<\infty$ and $\omega$ is a radial function restricted by the requirement that the space $A_{\omega}^p$ is complete.

This line of inquiry has been carried out on many spaces of analytic functions in recent years.  It is typically the case that the rotations of the disk give rise to isometric composition operators which are surjective on many function spaces over the disk and one particular interest is to determine when these are the \textit{only} symbols which induce isometries.  This is not the case for the Dirichlet space $\mathcal{D}$ as it is shown in \cite{MVD} that any univalent, full map of the disk which fixes the origin induces an isometric composition operator.  For the Bloch space $\mathcal{B}$, it is also the case that there are non-rotation symbols which induce isometric composition operators.  In \cite{FC} these symbols are classified by the Bloch semi-norm while in \cite{MVBl} they are classified in terms of the hyperbolic derivative of the symbol $\Ph$.  However, on the $\alpha$-Bloch spaces (or Bloch-type spaces), $\alpha>0$, $\alpha\neq 1$, it is the case that only rotations of the disk induce isometric composition operators; see \cite{NZ}.  These two results stand in contrast since the Bloch space is the $\alpha=1$ Bloch-type space.  Finally, on BMOA, the paper \cite{JL} demonstrates that there are non-rotation symbols which induce isometric composition operators.  Here the author classifies these maps in terms of involution automorphisms of the disk and provides some examples of such maps; specifically the author provides examples of non-inner functions which induce isometric composition operators on BMOA.

For our particular setting we will seek to classify the isometric composition operators acting on the analytic Besov spaces $B_p$, $1<p<\infty$, $p\neq 2$; the space $B_2$ is norm equivalent to the Dirichlet space $\mathcal{D}$.  This was also explored in \cite{MVB} for $2<p<\infty$ but with the additional assumption that the symbol $\Ph$ is univalent.  It was then shown that (among the univalent symbols) the only isometric composition operators arise from rotations of the disk.  However, the additional assumption that $\Ph$ is univalent is a strong restriction.  We will weaken this restriction and prove that the result still holds; it still remains to answer this question in the most general setting.  The authors in \cite{MVB} prompt their audience to also consider the case $1<p<2$.  We show in this case that only rotations of the disk induce isometric composition operators.  To do this we will use results of Kolaski mentioned above which classify general isometries on the weighted Bergman spaces.  We also investigate this same question for $B_p, 1<p<\infty, p\neq 2$, equipped with an equivalent norm and we show in this case that an isometric composition operator must be induced by a rotation of the disk.

As a final note, the Bloch space, BMOA, the Dirichlet space, and the Besov spaces are all M\"{o}bius invariant spaces; moreover, for $1<q<2<p<\infty$ the following containments are all strict, $$B_q\subseteq \mathcal{D}\subseteq B_p\subseteq BMOA\subseteq \mathcal{B}.$$ Our results and those mentioned earlier demonstrate that there is a nice contrast in symbols that induce isometric composition operators on this range of spaces.

\section{Preliminaries}
We set $\D=\{z\in\mathbb{C}:|z|<1\}$ to be the open unit disk in the complex plane and $H(\D)$ to be the set of functions analytic on $\D$. For $1<p<\infty$, the analytic Besov space $B_p$ is given by $$B_p(\D)=\left\{f \in H(\D):\|f\|_p^p= \int_{\D}|f'(z)|^p(1-|z|^2)^{p-2}\,dA(z)<\infty\right\},$$ where $dA$ is Lebesgue area measure on $\D$ normalized so that $A(\D)=1$. The quantity $\|f\|_p$ defines a semi-norm on $B_p$ and we obtain a Banach space norm by setting $$\|f\|_{B_p}=|f(0)|+\|f\|_p.$$  Note that $B_2$ is the classical Dirichlet space (this norm is equivalent to the classical norm but only provides a Banach space structure).  These spaces are particularly relevant in the study of function spaces as they form a natural scale of M\"{o}bius invariant spaces meaning that for every conformal automorphism $\Ph$ of $\D$ we have $\|f\circ\Ph\|_p=\|f\|_p$ . Note that this is not enough to conclude that every automorphism of $\D$ induces an isometric composition operator on $B_p$; we shall see below that only the rotation automorphisms have this property.  For general references on the Besov spaces we recommend \cite{AFP}, \cite{KZB}, and \cite{KZ}; for general information on composition operators acting on Besov spaces, see \cite{MT} and \cite{KS}.

In order to study isometric composition operators on the Besov spaces, we will need to understand the isometries of the weighted Bergman spaces.  For $\A>-1$ and $1\leq p<\infty$, the standard weighted Bergman space is defined as $$A_{\A}^p(\D)=\left\{f\in H(\D):\|f\|_{A_{\A}^p}^p=\int_\D|f(z)|^p(1-|z|^2)^{\A}\,dA(z)<\infty\right\}.$$  From this definition it is clear that $f\in B_p$ if and only if $f'\in A_{p-2}^p$ with \begin{equation}\label{eqn:norm}\|f\|_p=\|f'\|_{A_{p-2}^p}.\end{equation}  In \cite{CKW} it is shown that an isometry acting on a weighted Bergman space is given by a weighted composition operator and the following theorem appears there as Theorem 1 though we have modified it for our setting.

\begin{theorem}\label{thm:kolaski}
Let $1\leq p<\infty$ with $p\neq 2$ and $\A>-1$. If $T:A_{\A}^p\rightarrow A_{\A}^p$ is a linear isometry and $T1=\psi$, then there is an analytic function $\Ph$ which maps $\D$ onto an open, dense subset of $\D$ such that $$Tf=\psi(f\circ\Ph)$$ for all $f\in A_{\A}^p$; moreover, for every bounded Borel function $h$ on $\D$, $$\int_\D h(\Ph(z))|\psi(z)|^p(1-|z|^2)^{\A}\,dA(z)=\int_\D h(z)(1-|z|^2)^{\A}\,dA(z).$$
\end{theorem}

We make two important observations.  First, if $E\subseteq \D$ is a Borel set, then the characteristic function $\chi_E$ is a bounded Borel function and thus $$\int_\D \chi_E(\Ph(z))|\psi(z)|^p(1-|z|^2)^{\A}\,dA(z)=\int_\D \chi_E(z)(1-|z|^2)^{\A}\,dA(z).$$  This immediately implies that
\begin{equation}\label{eqn:intequality}\int_{\Ph^{-1}(E)}|\psi(z)|^p(1-|z|^2)^{\A}\,dA(z)=\int_E(1-|z^2|)^{\A}\,dA(z)\end{equation} for every Borel set $E\subseteq \D$ and where $\Ph^{-1}(E)=\{z\in \D: \Ph(z)=w\textup{ for some }w\in E\}$.

The second observation comes from investigating the proof of the above theorem (the statement above comes from \cite{CKW} but the full proof is given in \cite{CKR}).  There the author actually shows that $\Ph$ is a full map of the disk meaning that $A[\D\setminus\Ph(\D)]=0.$ This will be relevant in several calculations below.

\section{Isometric Composition Operators on $B_p$}

Throughout this section we will assume that $\CO$ is an isometry on $B_p$.  The first four results are general in the sense that they apply to $B_p$ for $1<p<\infty$ with $p\neq 2$; in fact the first three of these apply to $B_p$ for $1<p<\infty$.

\begin{lemma}\label{lem:seminorm}
Let $1<p<\infty$. If $\CO$ is an isometry on $B_p$, then $\Ph(0)=0$ and $\|f\circ \Ph\|_p=\|f\|_p$ for all $f\in B_p$.
\end{lemma}

\begin{proof}
To show that $\Ph(0)=0$, we first show that \begin{equation}\label{eqn:semisub}\|f\circ \Ph\|_p\leq\|f\|_p\end{equation} for all $f\in B_p$.  First suppose that $f\in B_p$ with $f(0)=0$.  Since $\CO$ is an isometry on $B_p$, we have $$|f(\Ph(0))|+\|f\circ\Ph\|_p=|f(0)|+\|f\|_p=\|f\|_p$$ and consequently $$\|f\circ\Ph\|_p=\|f\|_p-|f(\Ph(0)|\leq \|f\|_p.$$  If we now consider a general $f\in B_p$ with $f(0)=w_0$, it follows that the function $g(z)=f(z)-w_0$ is in $B_p$ and $g(0)=0$; hence $\|g\circ\Ph\|_p\leq \|g\|_p$.  Moreover, $\|g\|_p=\|f\|_p$ and $\|g\circ\Ph\|_p=\|f\circ\Ph\|_p$.  Thus $$\|f\circ\Ph\|_p=\|g\circ\Ph\|_p\leq \|g\|_p=\|f\|_p.$$

Now suppose that $\Ph(0)=a$ and let $\Ph_a$ be the involution automorphism of $\D$ which interchanges 0 and $a$. Next consider the map $\Ph_a\circ \Ph$.  It follows that $\Ph_a(\Ph(0))=0$ and $$|a|+\|\Ph_a\|_p=\|\Ph_a\|_{B_p}=\|\Ph_a\circ\Ph\|_{B_p}=\|\Ph_a\circ\Ph\|_p\leq \|\Ph_a\|_p$$ by (\ref{eqn:semisub}) and the fact that $\CO$ is an isometry.  Thus it must be the case that $a=0$.

To complete the proof, the equality $$|f(\Ph(0))|+\|f\circ\Ph\|_p=|f(0)|+\|f\|_p$$ together with the fact that $\Ph(0)=0$ implies that $$\|f\circ\Ph\|_p=\|f\|_p$$ for all $f\in B_p$ as desired.
\end{proof}

\begin{cor}\label{cor:automorphic}
Let $1<p<\infty$ and suppose $\Ph$ is an automorphism of $\D$. Then  $\CO$ is an isometry on $B_p$ if and only if $\Ph$ is a rotation of the disk.
\end{cor}

The next lemma is true on a variety of spaces and there are several proofs which can be given.  For our particular setting, one tactic is to use the fact that the adjoint of a surjective isometry is itself an isometry together with the fact that point evaluations are bounded linear functionals.  The proof we supply is of a different flavor but will provide the motivation for an interesting corollary.

\begin{lemma}\label{lem:ontounivalent}
Let $1<p<\infty$. If $\CO$ is a surjective isometry on $B_p$, then $\Ph$ is univalent.
\end{lemma}

\begin{proof}
Suppose to the contrary that $\Ph$ is not univalent.  Then there exist $a,b\in \D$ with $a\neq b$ and $\Ph(a)=\Ph(b)$.  Since $B^p$ contains univalent functions, it follows by the surjectivity of the operator that there exist functions $f, g\in B_p$ with $f\circ \Ph=g$, with $g$ univalent. However this implies that $g(a)=f(\Ph(a))=f(\Ph(b))=g(b)$ contradicting the fact that $g$ is univalent.
\end{proof}

The next result relates isometric composition operators on $B_p$ to isometric weighted composition operators on $A_{p-2}^p$. The connection between composition operators on $B_p$ and weighted composition operators on $A_{p-2}^p$ was first exploited in \cite{KS} to study boundedness and compactness of composition operators between $B_p$ and $B_q$.

\begin{lemma}\label{lem:weighted}
Let $1<p<\infty$ with $p\neq 2$.  If $\CO$ is an isometry on $B_p$, then the weighted composition operator $W_{\Ph',\Ph}$ is an isometry on $A_{p-2}^p$ and $\Ph$ is a full map of the disk, i.e. $A[\D\setminus\Ph(\D)]=0$.
\end{lemma}

\begin{proof}
Let $f\in A_{p-2}^p$ and choose $g\in B_p$ with $g'=f$ (any antiderivative will suffice). With this choice it is clear that $\|g\|_p=\|f\|_{A_{p-2}^p}$ by Eqn. (\ref{eqn:norm}).  Applying Lemma \ref{lem:seminorm},
$$\|W_{\Ph',\Ph}f\|_{A_{p-2}^p}^p=\|\Ph'(f\circ\Ph)\|_{A_{p-2}^p}^p=\|(g\circ\Ph)'\|_{A_{p-2}^p}^p=\|g\circ\Ph\|_p^p=\|g\|_p^p=\|f\|_{A_{p-2}^p}^p$$ which verifies that $W_{\Ph',\Ph}$ is an isometry on $A_{p-2}^p.$

The second conclusion now follows from the observation made at the end of Section 2 concerning the proof of Theorem \ref{thm:kolaski}.
\end{proof}

\subsection{$B_p$ with $1<p<2$}

For $\Ph$ an analytic self-map of $\D$ and $w\in \D$, we let $\{z_j(w)\}$ denote the sequence of zeros of the function $\Ph(z)-w$ and let $n_{\Ph}(w)$ denote the cardinality of the set $\Ph^{-1}(\{w\}).$ In \cite{HJMAPS}, the authors use $n_{\Ph}$ in their study of isometries on the Bergman spaces.  This function is lower semi-continuous for analytic maps by Rouche's Theorem, and is therefore measurable with respect to $A$.

\begin{prop}\label{prop:less2univalent}
Let $1<p<2$. If $\CO$ is an isometry on $B_p$, then $\Ph$ is univalent in $\D$.
\end{prop}

\begin{proof}
Our first goal will be to show that $n_{\Ph}(w)=1$ a.e. on $\D$. First note that Lemma \ref{lem:weighted} implies that $n_{\Ph}(w)\geq 1$ a.e. on $\D$. Let $B$ denote the set $$B=\{w\in \D: n_{\Ph}(w)>1\};$$ note that $B$ is a measurable set since $n_{\Ph}$ is a measurable function and that $n_{\Ph}(w)=1$ a.e. on $\D\setminus B$.  Let $f(z)=z$.  By Lemma \ref{lem:seminorm}, the Schwarz-Pick Lemma (with $p-2<0$), the change of variable $w=\Ph(z)$ (see \cite{CMC} Theorem 2.32), and Lemma \ref{lem:weighted}, we have $$\begin{aligned}\|f\|_p^p=\|\Ph\|_p^p&=\int_{\D}|\Ph'(z)|^p(1-|z|^2)^{p-2}\,dA(z)\\
&\geq\int_{\D}|\Ph'(z)|^2(1-|\Ph(z)|^2)^{p-2}\,dA(z)\\
&=\int_{\Ph(\D)}\sum_{j\geq 1}(1-|\Ph(z_j(w))|^2)^{p-2}\,dA(w)\\
&=\int_{\D}\sum_{j\geq 1}(1-|w|^2)^{p-2}\,dA(w)\\
&=\int_{\D\setminus B}(1-|w|^2)^{p-2}\,dA(w)+\int_B\sum_{j\geq 1}(1-|w|^2)^{p-2}\,dA(w)\\
&\geq \int_{\D\setminus B}(1-|w|^2)^{p-2}\,dA(w)+2\int_B(1-|w|^2)^{p-2}\,dA(w)\\
&=\int_{\D}(1-|w|^2)^{p-2}\,dA(w)+\int_B(1-|w|^2)^{p-2}\,dA(w)\\
&=\|f\|_p^p+\int_B(1-|w|^2)^{p-2}\,dA(w).\end{aligned}$$  From this it is evident that $A(B)=0$ and thus $n_{\Ph}(w)=1$ a.e. on $\D$.

The univalence of $\Ph$ is now a consequence of the fact that univalent functions are open maps.  A proof can be found in the proof of the main theorem in \cite{MVD}.
\end{proof}

\begin{theorem}\label{thm:less2isometry}
Let $1<p<2$. Then $\CO$ is an isometry on $B_p$ if and only if $\Ph$ is a rotation of the disk.
\end{theorem}

\begin{proof}
The fact that any rotation of the disk induces an isometric composition operator on $B_p$ is clear.  We therefore assume that $\CO$ is an isometry on $B_p$ and again consider the function $f(z)=z$. By Lemma \ref{lem:weighted} and Proposition \ref{prop:less2univalent} we have that $\Ph$ is a univalent, full map of the disk.  Using a calculation completely analogous to that used in the proof of the previous proposition for a univalent change of variable, we have
$$\begin{aligned}\|f\|_p^p=\|\Ph\|_p^p&=\int_{\D}|\Ph'(z)|^p(1-|z|^2)^{p-2}\,dA(z)\\
&\geq \int_{\D}|\Ph'(z)|^2(1-|\Ph(z)|^2)^{p-2}\,dA(z)\\
&=\int_{\Ph(\D)}(1-|w|^2)^{p-2}\,dA(w)\\
&=\int_{\D}(1-|w|^2)^{p-2}\,dA(w)\\
&=\|f\|_p^p
\end{aligned}$$ which implies that equality must hold throughout.  Specifically, the inequality in the second line must be an equality, or equivalently $$\int_{\D}\left(|\Ph'(z)|^{p-2}(1-|z|^2)^{p-2}-(1-|\Ph(z)|^2)^{p-2}\right)|\Ph'(z)|^2\,dA(z) =0.$$  The univalence of $\Ph$ guarantees us that $\Ph'$ does not vanish in $\D$.  Furthermore, since this integrand is non-negative by the Schwarz-Pick Lemma, it must be the case that $$|\Ph'(z)|^{p-2}(1-|z|^2)^{p-2}-(1-|\Ph(z)|^2)^{p-2}=0$$ in $\D$ or $$|\Ph'(z)|(1-|z|^2)=1-|\Ph(z)|^2$$ in $\D$.  Therefore $\Ph$ must be a disk automorphism, again by the Schwarz-Pick Lemma.  The fact that $\Ph(0)=0$ allows us to conclude that $\Ph$ must be a rotation of the disk.
\end{proof}

\subsection{$B_p$ with $p>2$}

For the case of surjective isometries, we have the following.

\begin{prop}\label{prop:greater2surjective}
Let $p>2$. Then $\CO$ is a surjective isometry on $B_p$ if and only if $\Ph$ is a rotation of the disk.
\end{prop}

This result follows immediately from Lemma \ref{lem:ontounivalent} and a computation analogous to that of Theorem \ref{thm:less2isometry}; see also \cite{MVB} Theorem 1.4.  The proof of Lemma \ref{lem:ontounivalent} also leads to a reduction in the general case.

\begin{cor}\label{cor:greater2reduction}
Let $p>2$ and suppose that $\CO$ is an isometry on $B_p$.  If the range of $\CO$ contains a univalent function, then $\Ph$ is a rotation of the disk.
\end{cor}
 
For a general isometry we are only able to obtain a partial result which relaxes the hypothesis of the analogous statement given in \cite{MVB}.

\begin{theorem}\label{thm:greater2isometry}
Let $p>2$ and suppose that $n_{\Ph}(w)=1$ a.e. in some neighborhood of the origin. Then $\CO$ is an isometry on $B_p$ if and only if $\Ph$ is a rotation of the disk.
\end{theorem}

\begin{proof}
If $\Ph$ is a rotation of the disk, then it is clear that $\CO$ is an isometry on $B_p$.  Assuming that $n_{\Ph}(w)=1$ a.e. in some neighborhood of the origin it follows that there is a neighborhood $U$ of $\Ph(0)=0$ such that $n_{\Ph}(w)=1$ a.e. on $U$.  Since $\Ph$ is continuous, $\Ph^{-1}(U)$ is an open set and thus there is an $0<r<1$ such that $\Ph(D(0,r))\subseteq U$, where $D(0,r)$ is the Euclidean disk centered at the origin with radius $r$. Set $E=\Ph(D(0,r))$; then $\Ph^{-1}$ is well-defined on $E$ as a single-valued inverse mapping and $\Ph^{-1}(E)=D(0,r)$.  By Lemma \ref{lem:weighted} and Eqn. (\ref{eqn:intequality}) we have (by techniques similar to the previous theorem) $$\begin{aligned}\int_{E}(1-|z|^2)^{p-2}\,dA(z)&=\int_{D(0,r)}|\Ph'(z)|^p(1-|z|^2)^{p-2}\,dA(z)\\
&\leq\int_{D(0,r)}|\Ph'(z)|^2(1-|\Ph(z)|^2)^{p-2}\,dA(z)\\
&= \int_{E}(1-|w|^2)^{p-2}\,dA(w).
\end{aligned}$$  This computation implies that equality must hold throughout and as in the proof of Theorem \ref{thm:less2isometry}, it must be the case that $\Ph$ is a rotation of the disk.
\end{proof}

We conjecture that every isometric composition operator on $B_p$ with $p>2$ must be induced by a rotation of the disk but the techniques we have employed seem unable to support this claim. Perhaps more sophisticated methods such as those used in \cite{HJMAPS} will prove more fruitful.

\section{Isometric Composition Operators on $B_p$ with an alternate norm}

As with many spaces of analytic functions, there are several equivalent norms for the analytic Besov spaces.  In this section we will classify the isometric composition operators with respect to one of these equivalent norms.  For the remainder of this section we fix $n\geq 2$.  In \cite{KZB} Zhu shows that $f\in B_p$ if and only if $$\int_\D|f^{(n)}(z)|^p(1-|z|^2)^{np-2}\,dA(z)<\infty;$$ in other words $f\in B_p$ if and only if the $n^{th}$ derivative of $f$ is in $ A_{np-2}^p$.  A norm equivalent to the one provided in Section 2 is obtain by setting \begin{equation}\label{eqn:equivnorm}\|f\|=\sum_{k=0}^{n-1}|f^{(k)}(0)|+\|f^{(n)}\|_{A_{np-2}^2}.\end{equation}

In \cite{HJ} the authors studied general isometries on the Besov spaces equipped with the norm just defined. There they proved the following result.

\begin{prop}[\cite{HJ} Corollary 2.1]\label{prop:permutation}
Let $1<p<\infty$ with $p\neq 2$ and let $n\geq 2$. If $T:B_p\rightarrow B_p$  is an isometry, then there exists a permutation $\pi$ of the set $\{0,1,2,\ldots, n-1\}$ and for each $k\in\{0,1,2,\ldots, n-1\}$ there is a unimodular constant $\lambda_k$ such that $$T\left(\frac{z^k}{k!}\right)=\lambda_k\frac{z^{\pi(k)}}{\pi(k)!}.$$
\end{prop}

Applying this theorem to the specific setting of composition operators we obtain the following.

\begin{theorem}\label{thm:equivnorm}
Let $1<p<\infty$ with $p\neq 2$ and let $n\geq 2$.  Then $\CO$ is an isometry on $B_p$ with the norm of Eqn. (\ref{eqn:equivnorm}) if and only if $\Ph$ is a rotation of the disk.
\end{theorem}

\begin{proof}
If $\Ph$ is a rotation of the disk, then a simple computation shows that $\CO$ is an isometry on $B_p$.  Assuming that $\CO$ is an isometry, we may apply the result of Proposition \ref{prop:permutation}.  Note that it suffices to show that $\pi(1)=1$ since this will imply that $$\Ph(z)=\CO z=\lambda_1 z,$$ where $\lambda_1$ is a unimodular constant.  First observe that for $k=0$, we have $$1=\CO1=\lambda_0\frac{z^{\pi(0)}}{\pi(0)!}$$ for all $z\in \D$ and where $\lambda_0$ is a unimodular constant.  If $\pi(0)\geq 1$, then $$\left|\lambda_0\frac{z^{\pi(0)}}{\pi(0)!}\right|<1$$ for all $z\in\D$; thus $\pi(0)=0$.

Now we set $\pi(1)=k\in\{1,2,3,\ldots,n-1\}$. Then we have $$\Ph(z)=\CO z=\lambda_1\frac{z^k}{k!},$$ where $\lambda_1$ is a unimodular constant.  Also, setting $\pi(2)=m\in\{1,2,3,\ldots,n-1\}\setminus\{k\}$, we see that $$\frac{\Ph(z)^2}2=\CO\left(\frac{z^2}{2}\right)=\lambda_2\frac{z^m}{m!},$$ where $\lambda_2$ is a unimodular constant, and it follows that $$\lambda_1^2\frac{z^{2k}}{2(k!)^2}=\lambda_2\frac{z^m}{m!}.$$  Since this equality holds in $\D$ (in $\overline{\D}$ in fact) it must therefore be the case that $2k=m$ and $$\frac{\lambda_1^2}{2(k!)^2}=\frac{\lambda_2}{m!}.$$  Taking the modulus of this last expression shows that $2(k!)^2=m!$ which implies that $2(k!)^2=(2k)!$.  Simplifying shows that $$k!=k(2k-1)(2k-2)\cdots(k+1)$$ which is impossible for $k\in\{2,3,\ldots,n-1\}$.  Thus it must be the case that $k=1$ and therefore $\Ph(z)=\lambda_1 z$ for some unimodular constant $\lambda_1$.
\end{proof}

\bibliographystyle{amsplain}
\bibliography{references.bib}
\end{document}